\documentclass[a4paper,11pt,oneside]{article}
\usepackage{amsfonts,amsmath,amsthm,scalerel,verbatim,amssymb}
\usepackage{pstricks-add}
\usepackage[normalem]{ulem}
\makeatletter
\newcommand{\vast}{\bBigg@{4}}
\newcommand{\Vast}{\bBigg@{5}}
\makeatother
\allowdisplaybreaks

\newcommand{\Z}{\mathbb{Z}}
\newcommand{\C}{\mathbb{C}}
\newcommand{\R}{\mathbb{R}}
\newcommand{\N}{\mathbb{N}}
\newcommand{\Q}{\mathbb{Q}}
\newcommand{\Res}{\mathrm{Res}}

\newcommand{\pres}[2]{\langle {#1}\ |\ {#2} \rangle}

\newtheorem{theorem}{Theorem}
\newtheorem{lemma}[theorem]{Lemma}
\newtheorem{prop}[theorem]{Proposition}
\newtheorem{corollary}[theorem]{Corollary}
\newtheorem{conjecture}[theorem]{Conjecture}
\newtheorem{problem}[theorem]{Problem}
\numberwithin{theorem}{section}
\numberwithin{equation}{section}

\newcommand{\onetwopres}[3]{\Bigg \langle {#1}\ \Bigg |\
\begin{array}{l}
{#2}\\{#3}
\end{array}\Bigg \rangle}
\newcommand{\onethreepres}[4]{\Bigg \langle {#1}\ \Bigg |\
\begin{array}{l}
{#2}\\{#3}\\{#4}
\end{array}\Bigg \rangle}

\newcommand{\onefivepres}[6]{\Vast \langle {#1}\ \Vast |\
\begin{array}{l}
{#2}\\{#3}\\{#4}\\{#5}\\{#6}
\end{array}\Vast\rangle}

\begin{document}
\title{Cyclically presented groups as Labelled Oriented Graph groups}
\author{Vanni Noferini\thanks{Vanni Noferini acknowledges support by an Academy of Finland grant (Suomen Akatemian p\"{a}\"{a}t\"{o}s 331240) and partial support by the Visiting Fellows Programme of the University of Pisa.} and Gerald Williams\thanks{Gerald Williams was supported for part of this project by Leverhulme Trust Research Project Grant RPG-2017-334.}}

\maketitle

\begin{abstract}
We use results concerning the Smith forms of circulant matrices to identify when cyclically presented groups have free abelianisation and so can be Labelled Oriented Graph (LOG) groups. We generalize a  theorem of Odoni and Cremona to show that for a fixed defining word, whose corresponding representer polynomial has an irreducible factor that is not cyclotomic and not equal to $\pm t$, there are at most finitely many $n$ for which the corresponding $n$-generator cyclically presented group has free abelianisation. We classify when Campbell and Robertson's generalized Fibonacci groups $H(r,n,s)$ are LOG groups and when the Sieradski groups are LOG groups. We prove that amongst Johnson and Mawdesley's groups of Fibonacci type, the only ones that can be LOG groups are Gilbert-Howie groups $H(n,m)$. We conjecture that if a Gilbert-Howie group is a LOG group, then it is a Sieradski group, and prove this in certain cases (in particular, for fixed $m$, the conjecture can only be false for finitely many $n$). We obtain necessary conditions for a cyclically presented group to be a connected LOG group in terms of the representer polynomial and apply them to the Prishchepov groups.
\end{abstract}

\noindent \textbf{Keywords:} cyclically presented group, Fibonacci group, Sieradski group, LOG group, knot group, Wirtinger presentation, circulant matrix, Smith form,  resultant.

\noindent \textbf{MSC:} 20F05 (primary);  57M05, 57M15, 11C20, 15B05, 15B36  (secondary).

\section{Introduction}

A \em Labelled Oriented Graph (LOG) \em consists of a finite graph (possibly with loops and multiple
edges) with vertex set $V$ and edge set $E$ together with three maps $\iota,\tau,\lambda: E \rightarrow V$ called the \em initial
vertex map, terminal vertex map, \em and \em labelling map, \em respectively. The LOG determines a corresponding
\em LOG presentation\em
\[ \pres{V}{\tau(e)^{-1}\lambda(e)^{-1}\iota(e)\lambda(e)\ (e\in E)}.\]
A group with a LOG presentation is called a \em LOG group \em \cite{HowieRibbonDiscComp85}. When the underlying graph is connected we have a \em connected LOG, \em a \em connected LOG presentation\em, and a \em connected LOG group\em. The abelianisation of a LOG group is free abelian with rank equal to the number of components of the LOG and so the abelianisation of a connected LOG group is the infinite cyclic group.

As is well known, the Wirtinger presentation of the fundamental group of a classical knot or link complement is a LOG presentation in which the number of components of the LOG is equal to the number of components of the link and so knot groups are connected LOG groups. In higher dimensions, the fundamental group of the complement of any smoothly embedded, closed, orientable, connected $k$-manifold ($k\geq 2$) in the $(k+2)$-sphere is a connected LOG group \cite{Simon}. In particular, any $k$-knot group (i.e.\,the fundamental group of a $k$-sphere $S^k$ in $S^{k+2}$) is a connected LOG group. Further examples of LOG groups include all right angled Artin groups and braid groups.

A \em cyclic presentation \em is a group presentation of the form
\begin{alignat*}{1}
P_n(w)=\pres{x_0,\ldots ,x_{n-1}}{w(x_i,x_{i+1},\ldots ,x_{i+n-1})\ (0\leq i<n)}
\end{alignat*}
where $n\geq 1$ and the \em defining word \em $w=w(x_0,x_1,\ldots ,x_{n-1})$ is some element of the free group $F(x_0,\ldots ,x_{n-1})$ where subscripts are taken $\bmod$ $n$ and the group $G_n(w)$ it defines is called a \em cyclically presented group. \em The cyclically presented groups that we consider in this article are the  \em Prishchepov groups \em
\begin{alignat*}{1}
P(r,n,k,s,q)=\pres{x_0,\ldots,x_{n-1}}{\prod_{j=0}^{r-1} x_{i+qj}=\prod_{j=0}^{s-1} x_{i+qj+(k-1)}\ (0\leq i<n)}
\end{alignat*}
where $n,r,s\geq 1$, $0\leq k, q<n$, that were introduced in \cite{Prishchepov}, and their special cases the \em generalized Fibonacci groups \em
\begin{alignat*}{1}
H(r,n,s)=\pres{x_0,\ldots,x_{n-1}}{\prod_{j=0}^{r-1} x_{i+j}=\prod_{j=0}^{s-1} x_{i+j+r}\ (0\leq i<n)}
\end{alignat*}
where $n,r,s\geq 1$ \cite{CampbellRobertson75}, the \em groups of Fibonacci type \em $G_n(m,k)=G_n(x_0x_mx_k^{-1})$ ($0\leq m,k<n$, $n \geq 1$) introduced in \cite{CHRsurvey,JohnsonMawdesley} (see \cite{Wsurvey} for a survey), and in particular the \em Gilbert-Howie groups \em $H(n,m)=G_n(x_0x_mx_1^{-1})$ of \cite{GilbertHowie} and the \em Sieradski groups \em $S(2,n)=G_n(x_0x_2x_1^{-1})$ of \cite{Sieradski}. We remark that the Prishchepov groups are precisely the groups of type $\mathfrak{F}$ of \cite{BW2} with non-positive defining word. (A word $w$ is \em positive \em if it does not involve the inverse of any generator.)

Connections between HNN extensions of cyclically presented groups and LOG groups have been investigated in \cite{GilbertHowie,HowieWilliams,SV}. An almost complete classification of groups $H(r,n,s)$ that are connected LOG groups was given in \cite{WilliamsLOG}; Chinyere and Bainson classify the perfect groups $H(r,n,s)$ \cite{BainsonChinyere}, completing the connected LOG groups classification. Asphericity of certain cyclic presentations of the form $P_n(x_0wx_1^{-1}w^{-1})$ that are (connected) \em Word Labelled Oriented Graph presentations \em (or \em Wirtinger presentations\em) is established in \cite[Section 3]{HarlanderRosebrock}. In this article we investigate when cyclically presented groups are LOG groups or connected LOG groups.

Any finitely generated abelian group $A$ is isomorphic to a group of the form $A_0\oplus \Z^\beta$ where $A_0$ is a finite abelian group and $\beta=\beta(A)\geq 0$ is the \em Betti number \em (or \em torsion-free rank\em) of $A$. Thus $A$ is infinite if and only if $\beta(A)\geq 1$ and $A$ is a free abelian group if and only if $A_0=1$. Given a group presentation $P=\pres{x_0,\ldots,x_{n-1}}{R_0,\ldots ,R_{m-1}}$ ($n,m\geq 1$) the \em relation matrix \em of $P$ is the
$n\times m$ integer matrix $M$ whose $(i,j)$ entry is the exponent sum of generator $x_i$ in relator $R_j$. If the rank of $M$ is $r$ and the invariant factors of the Smith Form of $M$ are $s_1,\ldots,s_{r}$ then the abelianization of the group $G$ defined by the presentation $P$ is
\[ G^\mathrm{ab}\cong \Z_{s_1}\oplus \ldots \oplus \Z_{s_{r}}\oplus \Z^{n-r};\]
see, for example, \cite[page 146--149, Theorem 3.6]{MKS} or \cite[pages 54--57, Theorem 5]{JohnsonBook}. Thus $\beta(G^\mathrm{ab})=n-r$ and if $G^\mathrm{ab}=A_0\oplus \Z^\beta$ we have $|A_0|=|\prod_{i=1}^r s_i|$, i.e.\,the last non-zero determinantal divisor of $M$, which we denote by $\gamma_r$.

Returning to cyclically presented groups, for each $0\leq i<n$, we write $a_i$ to denote the exponent sum of $x_i$ in $w(x_0,\ldots ,x_{n-1})$. Then the relation matrix of $P_n(w)$ is the circulant matrix $C$ whose first row is $(a_0,a_1,\ldots, a_{n-1})$. The \em representer polynomial \em of $C$ is the polynomial
\[ f(t)=\sum_{i=0}^{n-1}a_it^i \in \Z[t]\]
and we set $g(t)=t^n-1\in \Z[t]$. Given such an $f\in \Z[t]$ we say that $C$ is the $n\times n$ circulant matrix \em associated with \em  $f$. It is well known, and much used in work on cyclically presented groups, that the order of the abelianisation of a cyclically presented group can be expressed as a resultant
\[ |G_n(w)^\mathrm{ab}|=|\mathrm{det}(C)|=|\prod_{\theta^n=1} f(\theta)| = |\mathrm{Res}(f,g)|\]
if this is non-zero, and $G_n(w)^\mathrm{ab}$ is infinite otherwise \cite{JohnsonBook}. As we will only be interested in the absolute values of resultants (and not the sign), to avoid repetitive use of modulus signs we will take $\mathrm{Res}(\cdot , \cdot)$ to mean $|\mathrm{Res}(\cdot , \cdot)|$ throughout this article. Thus we have the following criterion for $G_n(w)$ to be a perfect group:
\begin{alignat}{1}
G_n(w)^\mathrm{ab}=1\Leftrightarrow \mathrm{Res}(f,g)=1.\label{eq:perfect}
\end{alignat}
In particular, if $w$ is positive then $G_n(w)$ is perfect if and only if $w$ has length 1, in which case $G_n(w)=G_n(x_0)=1$.

Results of \cite{NoferiniWilliamsCompanion} allow information about the Smith form of the circulant matrix $C$ to be obtained from the polynomials $f(t),g(t)$, and so reveal structural information about $G_n(w)^\mathrm{ab}$. The following theorem gives a formula for the rank $\rho$ and last non-zero determinantal divisor $\gamma_\rho$ of $C$. Below and throughout this article, given polynomials $p(t),q(t)\in \Z[t]$ we write $(p(t),q(t))$ to denote the monic greatest common divisor of $p(t)$ and $q(t)$.

\begin{theorem}[{\cite[Theorem A and Corollary B]{NoferiniWilliamsCompanion}}]\label{thm:SNFresultant}
  Let $f(t)\in \Z[t]$, $g(t)=t^n-1$ and let $f(t)=F(t)z(t), g(t)=G(t)z(t)$ where $z(t)=(f(t),g(t)) \in \Z[t]$ and let $C$ be the $n\times n$ circulant matrix associated with $f$. If the Smith normal form of $C$ is the matrix $\mathrm{diag}_n(s_1,\ldots ,s_\rho,0,\ldots ,0)$ (so that $\rho=\mathrm{rank} (C)$) then $\rho=n-\mathrm{deg} (z(t))$ and the last non-zero determinantal divisor
  \[ \gamma_\rho = \prod_{i=1}^\rho s_i = \mathrm{Res}(F,G).\]
\end{theorem}

\begin{corollary}\label{cor:G_n(w)^ab}
Let $G_n(w)$ be a cyclically presented group with representer polynomial $f(t)$, and let $g(t)=t^n-1$. Then $G_n(w)^\mathrm{ab}\cong A_0\oplus \Z^\rho$ where $A_0$ is a finite abelian group of order $\gamma_\rho$, where $\rho$ and $\gamma_\rho$ are as given in Theorem \ref{thm:SNFresultant}.
\end{corollary}

The following immediate consequence of Corollary \ref{cor:G_n(w)^ab} should be compared to the necessary and sufficient condition for $G_n(w)$ to be perfect, given at (\ref{eq:perfect}):
\begin{alignat}{1}
G_n(w)^\mathrm{ab}~\mathrm{is~free~abelian}\Leftrightarrow \mathrm{Res}(F,G)=1.\label{eq:freeabelian}
\end{alignat}
Since LOG groups have free abelianisation, condition (\ref{eq:freeabelian}) will be an important tool for us.

In Section \ref{sec:OdoniCremona} we recall a problem (Problem \ref{prob:OdoniCremona}) posed by Odoni \cite{Odoni} and Cremona \cite{Cremona}, namely: given a defining word $w$, to determine all values of $n$ for which the corresponding cyclically presented group $G_n(w)$ is perfect (i.e.\,is free abelian of rank~0). We also recall a theorem of theirs (Corollary \ref{cor:OdoniCremonafinitelymany}) which states that for a fixed defining word $w$ whose corresponding representer polynomial has an irreducible, non-constant, non-cyclotomic, factor that is not equal to $\pm t$, there can be at most finitely many values of $n$ for which $G_n(w)$ is perfect. In Problem \ref{prob:FreeAbelian} we generalize this to consider cyclically presented groups whose abelianisation is free abelian (of arbitrary rank) and in Corollary \ref{cor:finitelymanyfreeabelian} we obtain a result that is analogous to, and generalizes, Corollary \ref{cor:OdoniCremonafinitelymany} which (in Theorem \ref{thm:finitelymanyfreeabelianGH}) we later apply to Gilbert-Howie groups.

In Section \ref{sec:Hrns} we consider the groups $H(r,n,s)$. In Theorem \ref{thm:HrnsSNF} and Corollary \ref{cor:Hrns} we obtain information about $H(r,n,s)^\mathrm{ab}$ and in Corollary \ref{cor:HrnsfreeabelianLOG} we extend the classification of groups $H(r,n,s)$ that are connected LOG groups \cite{WilliamsLOG,BainsonChinyere} to classify all groups $H(r,n,s)$ that are LOG groups. In Section \ref{sec:CHR} we turn our attention to the groups of Fibonacci type $G_n(m,k)$ and in Theorem \ref{thm:stronglyyirredCHRnotLOG} we show that if a group $G_n(m,k)$ is a LOG group then it is isomorphic to a Gilbert-Howie group $H(n,m)$, and we consider these groups in Section \ref{sec:GH}. In Theorem \ref{thm:sieradskiasLOG} and Corollary \ref{cor:sieradski} we show that the Sieradski group $S(2,n)=H(n,2)$ is a LOG group if and only if $6|n$. We conjecture (Conjecture \ref{conj:H(n,m)freeabelian}) that the groups $S(2,n)$ with $6|n$ are the only cases when $H(n,m)^\mathrm{ab}$ is free abelian, and hence that these are the only cases when $H(n,m)$ is a LOG group. In support of this conjecture, Theorem \ref{thm:finitelymanyfreeabelianGH} shows that for fixed $m\geq 3$ there are most finitely many $n$ for which $H(n,m)^\mathrm{ab}$ is free abelian. Theorem \ref{thm:GilbertHowieFreeAbelianmod6b} provides further support for the conjecture and Corollary \ref{cor:H(n,m)freeabeliann=6bor12bor24b} proves it when $n=6b,12b$ or $24b$ where $(b,6)=1$. In Section \ref{sec:knotgroups} we consider when a cyclically presented group $G_n(w)$ is a connected LOG group. In Theorem \ref{thm:G^ab=Z} we give necessary conditions on the representer polynomial $f(t)$ for $G_n(w)^\mathrm{ab}$ to be isomorphic to $\Z$ (a necessary condition for $G_n(w)$ to be a connected LOG group) to hold. In Corollaries \ref{cor:positiveword} and \ref{cor:prishchepov} we apply this to cyclically presented groups with positive defining words, and to the Prishchepov groups $P(r,n,k,s,q)$.

\section{Groups $G_n(w)$ with free abelianisation for at most finitely many $n$}\label{sec:OdoniCremona}

A problem from the theory of cyclically presented groups is to determine, given a defining word $w$, the values of $n$ for which the corresponding group $G_n(w)$ is perfect. Using (\ref{eq:perfect}) this translates to the following Diophantine problem:

\begin{problem}[{\cite[Problem B]{Odoni},\cite[Problem B]{Cremona}}]\label{prob:OdoniCremona}
Given $f(t)\in \Z[t]$ and $g(t)=t^n-1$ determine all $n\in \N$ such that $\mathrm{Res}(f,g)=1$.
\end{problem}
The following partial answer was provided in \cite{Odoni,Cremona}:

\begin{theorem}[{\cite[Theorem 1(ii)]{Odoni},\cite[Proposition 1]{Cremona}}]\label{thm:OdoniCremonaFinitelyMany}
Let $f(t)\in \Z[t]$ be a non-constant, irreducible polynomial, that is not cyclotomic, and $f(t)\neq \pm t$, and let $g(t)=t^n-1$. Then there exist at most finitely many integers $n$ for which $\mathrm{Res}(f,g)=1$.
\end{theorem}

In fact, \cite[Theorem 1(ii)]{Odoni} proves somewhat more, as  its hypotheses also allow for $f$ to be a cyclotomic polynomial $\Phi_m$ for many values of $m$. The following formulation is convenient for applications (compare \cite[Theorem 1]{Cremona}):

\begin{corollary}\label{cor:OdoniCremonafinitelymany}
Let $f(t)\in \Z[t]$ have at least one irreducible factor $h(t)$ that is non-constant, not cyclotomic, and $h(t)\neq \pm t$, and let $g(t)=t^n-1$. Then there exist at most finitely many integers $n$ for which $\mathrm{Res}(f,g)=1$.
\end{corollary}

A generalisation of the problem considered above is to determine, given a defining word $w$, the values of $n$ for which the corresponding group $G_n(w)$ has free abelianisation. Using (\ref{eq:freeabelian}) this translates to the following:

\begin{problem}\label{prob:FreeAbelian}
Given $f(t)\in \Z[t]$ and $g(t)=t^n-1$ determine all $n\in \N$ such that $\mathrm{Res}(F,G)=1$, where $F(t)=f(t)/z(t)$, $G(t)=g(t)/z(t)$, where $z(t)=(f(t),g(t))$.
\end{problem}

In Corollary \ref{cor:finitelymanyfreeabelian} we provide a partial answer to Problem \ref{prob:FreeAbelian} that is analogous to, and generalizes, Corollary \ref{cor:OdoniCremonafinitelymany}. This is a corollary to the following refinement of Theorem \ref{thm:OdoniCremonaFinitelyMany}, whose proof extends the proof of (the corresponding part of) \cite[Theorem 1(ii)]{Odoni}. Recall that, for a nonconstant polynomial $f(t)\in \C[t]$ with leading coefficient $l$ the \em Mahler measure \em $\mathcal{M}(f)$ \cite[p. 271]{BorweinErdelyi} is defined as

\[ \mathcal{M}(f)= |l| \mathop{\prod_{f(\theta)=0,}}_{|\theta|>1}  |\theta|.\]

\begin{theorem}\label{thm:Mahler}
Let $h(t)\in \Z[t]$ be an irreducible polynomial of degree $m\geq 1$, that is not cyclotomic, with $h(t)\neq \pm t$ and let $g(t)=t^n-1$. Then there exist real constants $c,d>0$, depending on $m$ and $h(t)$ but not on $n$, such that the resultant $\mathrm{Res}(h,g)\geq c\mu^n n^{-d}$, where  $\mu=\mathcal{M}(h)>1$. In particular, there are at most finitely many integers for which $\mathrm{Res}(h,g)=1$.
\end{theorem}

For the proof of Theorem \ref{thm:Mahler} we need the following technical lemma:

\begin{lemma}\label{lem:Taylor}
Let $0\leq \epsilon \leq 1$ and suppose that $z \in \C$ satisfies $|z|>\epsilon$ and $-\pi < \Im (z) \leq \pi$. Then $|e^z-1|> \epsilon /2$.
\end{lemma}


\begin{proof}
Without loss of generality we may assume $\epsilon>0$. Write $z=x+iy$ for $x,y \in \R$ and let $f(x,y)=e^{2x}+1-2e^x \cos(y)$. Then $|e^z-1|=\sqrt{f(x,y)}$ and $f(x,y) \geq (e^x-1)^2$. If $|x| > \epsilon/\sqrt{2}$ then, since $(e^x-1)^2$ is decreasing for $x<0$ and increasing for $x>0$, we have
$|e^z-1| \geq 1 - e^{-\epsilon/\sqrt{2}} > \epsilon/2$. On the other hand if $|x| \leq \epsilon/\sqrt{2}$ then $|y| > \epsilon/\sqrt{2}$ and hence, since $-\pi <y\leq \pi$, we have
$f(x,y) > e^{2x}+1 - 2e^x(1-\epsilon^2/2)$ which (by minimising this function) is bounded below by $\epsilon^2 - \epsilon^4/4$. Hence $|e^z-1| > \epsilon \sqrt{3}/{2}  > \epsilon/2$.
\end{proof}

\begin{proof}[Proof of Theorem \ref{thm:Mahler}]
We first recall a classical result, due to Kronecker,  showing that the assumptions on $h(t)$ imply $\mu=\mathcal{M}(h)>1$. Indeed, suppose not, so $h(t)$ is monic and has no roots outside the unit circle. Then, since the product of the absolute values of the roots is equal to $|h(0)| \geq 1$, all the roots lie on the unit circle. But every unimodular algebraic number whose conjugates over $\Q$ are all unimodular is a root of unity, (see, for example, \cite[Lemma 1.2]{Odoni}) and so we conclude that $h(t)$ is cyclotomic, a contradiction.

Without loss of generality we may assume $n>m$. Let $\alpha_i$ ($i=1,\ldots , m_O$) be the $m_O$ roots of $h$ outside the unit circle, let $\beta_j$ ($j=1,\ldots , m_I$) be the $m_I$ roots of $h$ inside the unit circle, and let $\gamma_k$ ($k=1,\dots,m_C$) be the $m_C$ roots of $h$ on the unit circle, where $m_O+m_I+m_C=m$. Then each $|\alpha_i^{-1}|<1, |\beta_j|<1, |\gamma_k|=1$. Moreover, if $l$ is the leading coefficient of $h$,
\begin{alignat*}{1}
\mathrm{Res}(h,g)
&=  |l|^n \prod_{i=1}^{m_O} |\alpha_i^n-1| \cdot \prod_{j=1}^{m_I} |\beta_j^n-1| \cdot \prod_{k=1}^{m_C} |\gamma_k^n-1|  \\
&= |l|^n\left(\prod_{i=1}^{m_O} |\alpha_i|^n \right) \cdot \prod_{i=1}^{m_O} |1-\alpha_i^{-n}| \cdot \prod_{j=1}^{m_I} |\beta_j^n-1|  \cdot \prod_{k=1}^{m_C} |\gamma_k^n-1|\\
&=\mu^n \cdot \prod_{i=1}^{m_O} |1-\alpha_i^{-n}| \cdot \prod_{j=1}^{m_I} |\beta_j^n-1| \cdot \prod_{k=1}^{m_C} |\gamma_k^n-1|
\intertext{so}
\log (\mathrm{Res}(h,g)) &=
n\log(\mu) + \sum_{i=1}^{m_O} \log |1-\alpha_i^{-n}| + \sum_{j=1}^{m_I} \log|\beta_j^n-1| + \sum_{k=1}^{m_C} \log |\gamma_k^n-1|.
\end{alignat*}
Now, by definition, there exist $r<1$ and $R>1$ such that $|\alpha_i|>R$ and $|\beta_j|<r$ for each $1\leq i\leq m_O$, $1\leq j\leq m_I$. Now define $\delta =\max \{r,R^{-1}\}$. Then
\begin{alignat*}{1}
\log (\mathrm{Res}(h,g))
&\geq n\log(\mu) + \sum_{i=1}^{m_O} \log (1-\delta^n) + \sum_{j=1}^{m_I} \log(1-\delta^n)+ \sum_{k=1}^{m_C} \log |\gamma_k^n-1| \\
&= n\log(\mu) + (m_O+m_I) \log (1-\delta^n)+ \sum_{k=1}^{m_C} \log |\gamma_k^n-1|\\
&\geq n\log(\mu) + m \log (1-\delta^m)+ \sum_{k=1}^{m_C} \log |\gamma_k^n-1|
\end{alignat*}
since $n>m$. To estimate $|\gamma_k^n-1|$ we observe that since $\gamma_k$ is not a root of unity (otherwise $h$ would be cyclotomic) by Baker's theorem \cite[Theorem 3.1]{BakerBook} $|n\log(\gamma_k)|>n^{-C_k}$ where $C_k$ is a constant depending only on $\gamma_k$ and $\log (\cdot)$ denotes the principal branch of the logarithm. Setting $\epsilon=n^{-C_k}$, $z=n\log \gamma_k$ then $0\leq \epsilon\leq 1$ so Lemma \ref{lem:Taylor} implies $|\gamma_k^n-1|> n^{-C_k}/2$ and, denoting $C_\mathrm{max}=\max_{1 \leq k \leq m_C} C_k$ (note that $C_{\mathrm{max}}$ depends on $h(t)$),
\[  \log (\mathrm{Res}(h,g)) \geq n\log(\mu) + m \log (1-\delta^m) - m C_\mathrm{max} \log n - m \log 2 .    \]
Setting $c=[(1-\delta^m)/2]^m$ and $d=m C_\mathrm{max}$ yields the statement.
\end{proof}

\begin{corollary}\label{cor:finitelymanyfreeabelian}
Let $f(t)\in \Z[t]$ have at least one irreducible factor $h(t)$ that is non-constant, not cyclotomic, and $h(t)\neq \pm t$. For all $n$, let $g(t)=t^n-1$, $F(t)=f(t)/z(t)$, $G(t)=g(t)/z(t)$, where $z(t)=(f(t),g(t))$. Then, for any positive integer $k$, there are at most finitely many integers $n$ for which $\mathrm{Res}(F,G) \leq k$.
\end{corollary}

\begin{proof}
Let $S=\{d\ |\ \Phi_d~\mathrm{divides}~f\}$. Then, for all $n\geq 1$, $z(t)=\prod_{d\in S'} \Phi_d(t)$ for some $S'\subseteq S$. Let $R=\max_{S'\subseteq S} \bigg\{ \mathrm{Res} \left( \frac{f}{\prod_{d\in S'} \Phi_d}, \prod_{d\in S'} \Phi_d \right) \bigg\}$, a constant. Then $\mathrm{Res}(F,z)\leq R$ for all $n\geq 1$.

Noting that $h(t)$ is a factor of $F(t)$, by Theorem \ref{thm:Mahler} there exist constants $c,d>0$ such that
\[   \Res(F,G)=\frac{\Res(F,g)}{\Res(F,z)} \geq \frac{\Res(h,g)}{M} \geq \frac{c \mu^n}{n^d R} > k \]
(where $\mu=\mathcal{M}(h)>1$) for any sufficiently large $n$.
\end{proof}

\section{Generalized Fibonacci groups $H(r,n,s)$ as LOG groups}\label{sec:Hrns}

The representer polynomial of $H(r,n,s)$ is
\[f^{r,s}(t)=1+t+t^2 +\ldots +t^{r-1} -t^r(1+t+t^2+\ldots +t^{s-1}).\]

For $n,r,s\geq 1$ we define $d=(r,n,s)$, $R=r/d, N=n/d, S=s/d$.
In Theorem \ref{thm:HrnsSNF} we calculate the last non-zero determinantal divisor $\gamma_\rho$ of the $n\times n$ circulant matrix associated with $f^{r,s}(t)$ for all $r,n,s$ and in Corollary \ref{cor:Hrns} we relate the abelianisation $H(r,n,s)^\mathrm{ab}$ to $H(R,N,S)^\mathrm{ab}$. Note that by inverting the relators, replacing each generator by its inverse, and negating the subscripts $H(r,n,s)\cong H(s,n,r)$ so we may assume $s\geq r$.

\begin{theorem}\label{thm:HrnsSNF}
  Let $n,r,s,\geq 1$, $d=(r,n,s)$, $R=r/d, N=n/d, S=s/d$ and let
  \[f^{r,s}(t)=1+t+t^2 +\ldots +t^{r-1} -t^r(1+t+t^2+\ldots +t^{s-1}).\]
  \begin{itemize}
    \item[(a)] If $s>r$ then $\rho=n-d+1$ and
    \[ \gamma_\rho =\frac{\mathrm{Res}(f^{R,S}(t),t^N-1)^d}{(S-R)^{d-1}}.\]

    \item[(b)] If $s=r$ then $\rho=n-d$ and $\gamma_\rho=N^{d-1}$.
  \end{itemize}
\end{theorem}

\begin{proof}
(a) By \cite[Proof of Theorem C]{WilliamsLOG}
\begin{alignat*}{1}
z(t)&=1+t+t^2+\ldots +t^{d-1},\\
F(t)
&= (1-t^r)(1+t^d+\ldots +t^{(R-1)d})-t^{2r}(1+t^d+\ldots +t^{(S-R-1)d})\\
&= 1+t^d+\ldots +t^{(R-1)d}-t^{dR}(1+t^d+\ldots +t^{(S-1)d}),\\
G(t)&=(1-t)(1+t^d+\ldots +t^{(N-1)d})
\end{alignat*}
therefore (and as shown in \cite[Proof of Theorem C]{WilliamsLOG}) $\rho=n-d+1$. By Theorem \ref{thm:SNFresultant}
\begin{alignat*}{1}
\gamma_\rho
&=\mathrm{Res}(F,G)\\
&=\mathrm{Res}(1+t^d+\ldots +t^{(R-1)d}-t^{dR}(1+t^d+\ldots +t^{(S-1)d}),(1-t))\cdot\\
&\quad \mathrm{Res}(1+t^d+\ldots +t^{(R-1)d}-t^{dR}(1+t^d+\ldots +t^{(S-1)d}),1+t^d+\ldots +t^{(N-1)d})\\
&=(S-R)\cdot \mathrm{Res}(1+t^d+\ldots +t^{(R-1)d}-t^{dR}(1+t^d+\ldots +t^{(S-1)d}),1+t^d+\ldots +t^{(N-1)d})\\
&=(S-R)\cdot \left(\mathrm{Res}(1+t+\ldots +t^{(R-1)}-t^{R}(1+t+\ldots +t^{(S-1)}),1+t+\ldots +t^{(N-1)})\right)^d\\
&=(S-R)\cdot \left(\frac
{\mathrm{Res}(1+t+\ldots +t^{(R-1)}-t^{R}(1+t+\ldots +t^{(S-1)}),t^N-1)}
{\mathrm{Res}(1+t+\ldots +t^{(R-1)}-t^{R}(1+t+\ldots +t^{(S-1)}),t-1)}
\right)^d\\
&=(S-R)\cdot \left(\frac
{\mathrm{Res}(1+t+\ldots +t^{(R-1)}-t^{R}(1+t+\ldots +t^{(S-1)}),t^N-1)}
{S-R}
\right)^d
\end{alignat*}
as required.

\smallskip

(b)
By \cite[Proof of Theorem C]{WilliamsLOG}
\begin{alignat*}{1}
z(t)&=1-t^d,\\
F(t)
&= (1+t+t^2+\ldots +t^{r-1})(1+t^d+\ldots +t^{(R-1)d})\\
&= (1+t+t^2+\ldots +t^{d-1})(1+t^d+\ldots +t^{(R-1)d})^2,\\
G(t)&=(1+t^d+\ldots +t^{(N-1)d}),
\end{alignat*}
therefore (and as shown in \cite[Proof of Theorem C]{WilliamsLOG}) $\rho=n-d$. By Theorem \ref{thm:SNFresultant}
\begin{alignat*}{1}
\gamma_\rho
&=\mathrm{Res}(F,G)\\
&=\mathrm{Res}((1+t+t^2+\ldots +t^{d-1}),(1+t^d+\ldots +t^{(N-1)d}))\cdot\\
&\quad \mathrm{Res}((1+t^d+\ldots +t^{(R-1)d}),(1+t^d+\ldots +t^{(N-1)d}))^2\\
&=N^{d-1}\cdot \mathrm{Res}((1+t+\ldots +t^{(R-1)}),1+t+\ldots +t^{(N-1)})^{2d}.
\end{alignat*}
But
\begin{alignat*}{1}
\mathrm{Res}(1+t+\ldots +t^{(R-1)},1+t+\ldots +t^{(N-1)})
&= \mathrm{Res}\left( \prod_{d|R, d>1} \Phi_d, \prod_{\delta |N, \delta>1} \Phi_\delta \right)\\
&= \prod_{d|R, d>1} \prod_{\delta |N, \delta>1} \mathrm{Res}\left( \Phi_d, \Phi_\delta \right).
\end{alignat*}
Now $(R,N)=1$ so in this last product $(d,\delta)=1$ so each $\mathrm{Res}\left( \Phi_d, \Phi_\delta\right)=1$ by \cite[Theorem 3]{Apostol}, and the result follows.
\end{proof}

\begin{corollary}\label{cor:Hrns}
 Let $n,r,s \geq 1$, $d=(r,n,s)$, $R=r/d, N=n/d, S=s/d$.
 \begin{itemize}
   \item[(a)] If $s\neq r$ then $H(r,n,s)^\mathrm{ab} \cong A_0 \oplus \Z^{d-1}$ where $A_0$ is a finite abelian group of order $|H(R,N,S)^\mathrm{ab}|^d/|S-R|^{d-1}$.
   \item[(b)] If $s=r$ then $H(r,n,s)^\mathrm{ab} \cong A_0 \oplus \Z^{d}$ where $A_0$ is a finite abelian group of order $N^{d-1}$.
 \end{itemize}
\end{corollary}

The group $H(r,n,s)$ is perfect if and only if $|r-s|=1$ and either $r\equiv 0$ or $s\equiv 0\bmod$ $n$ \cite[Theorem A]{BainsonChinyere}. This classification yields the following corollary, which classifies when $H(r,n,s)^\mathrm{ab}$ is free abelian and when $H(r,n,s)$ is a LOG group (recalling that free groups and knot groups are LOG groups), thus extending \cite[Theorem A]{WilliamsLOG}, \cite[Corollary B]{BainsonChinyere} to the (possibly) disconnected case.

\begin{corollary}\label{cor:HrnsfreeabelianLOG}
 Let $n,r,s\geq 1$, $d=(r,n,s)$, $R=r/d, N=n/d, S=s/d$.
 \begin{itemize}
   \item[(a)] If $s\neq r$ then $H(r,n,s)^\mathrm{ab}$ is free abelian if and only if $|r-s|=d$ and either $r\equiv 0\bmod$ $n$ or $s\equiv 0\bmod$ $n$, in which case $H(r,n,s)$ is free of rank $d-1$.
   \item[(b)] If $s=r$ then $H(r,n,s)^\mathrm{ab}$ is free abelian if and only if either
   \begin{itemize}
     \item[(i)] $n|r$, in which case $H(r,n,s)$ is free of rank~$n$; or
     \item[(ii)] $d=1$, in which case $H(r,n,s)$ is isomorphic to the fundamental group of the $(r,n)$ torus knot.
   \end{itemize}
 \end{itemize}
\end{corollary}

\begin{proof}
(a) If $H(r,n,s)^\mathrm{ab}$ is free abelian then Corollary \ref{cor:Hrns} implies that $H(R,N,S)$ is perfect. Then by \cite[Theorem A]{BainsonChinyere} $|R-S|=1$ and either $R\equiv 0\bmod$ $N$ or $S\equiv 0\bmod$ $N$. Equivalently, $|r-s|=d$ and $r\equiv 0\bmod$ $n$ or $s\equiv 0\bmod$ $n$. Now if $r\equiv 0\bmod$ $n$ or $s\equiv 0\bmod$ $n$ then by \cite[Lemma 11]{WilliamsLOG} the group $H(r,n,s)$ is isomorphic to the free product of $\Z_{|r-s|/(n,r-s)}$ and the free group of rank $(n,|r-s|)-1$. But $|r-s|/(n,r-s)=|r-s|/d=1$ and $(n,|r-s|)-1=d-1$, so $H(r,n,s)$ is free of rank $d-1$, as required.

(b) All but the `in which case' statements follow immediately from Corollary \ref{cor:Hrns}, and those statements follow from \cite[Lemmata 4 and 11]{WilliamsLOG}.
\end{proof}

\section{Groups of Fibonacci type as LOG groups}\label{sec:CHR}

A group of Fibonacci type $G_n(m,k)$ is called \em irreducible \em if $(n,m,k)=1$, and \em strongly irreducible \em if it is irreducible and $(n,k)>1$, $(n,m-k)>1$. (This definition is essentially the one given in \cite{BardakovVesnin}, though we omit the additional condition that $0<m<k<n$, which is unnecessary for our purposes.) The irreducibility condition prevents $G_n(m,k)$ decomposing as a free product in a canonical way \cite[Lemma 1.2]{BardakovVesnin}; moreover, as shown in \cite[Lemma 1.3]{BardakovVesnin}, if $G_n(m,k)$ is irreducible but not strongly irreducible then it is isomorphic to some Gilbert-Howie group $H(n,m')$. Thus, in considering the class of groups of Fibonacci type it suffices to consider the strongly irreducible groups $G_n(m,k)$ and the Gilbert-Howie groups $H(n,m)$. In this section we show that strongly irreducible groups $G_n(m,k)$ are not LOG groups (Theorem \ref{thm:stronglyyirredCHRnotLOG}) and in Section \ref{sec:GH} we consider Gilbert-Howie groups.

The representer polynomial of $G_n(m,k)$ is the trinomial $f(t)=t^m-t^k+1$. Throughout this section we let $\gamma_\rho$ denote the last non-zero determinantal divisor of the $n\times n$ circulant matrix $C$ associated with $f$. The following result classifies the groups $G_n(m,k)$ with infinite abelianisation and classifies the perfect groups $G_n(m,k)$. This was proved in \cite{Odoni} for the case $k=1$ and extended to the case $k>1$ in \cite{WilliamsCHR,WilliamsIJNT}.

\begin{theorem}[{\cite[Theorem 2]{Odoni},\cite[Theorem 4]{WilliamsCHR},\cite{WilliamsIJNT}}]\label{thm:WilliamsOdoni}
Let $n,m,k\geq 1$ and suppose $(n,m,k)=1$.
\begin{itemize}
  \item[(a)] $\mathrm{Res}(t^m-t^k+1,t^n-1)=0$ if and only if $n\equiv 0\bmod$ $6$ and $m\equiv 2k\bmod$ $6$;
  \item[(b)] $\mathrm{Res}(t^m-t^k+1,t^n-1)=1$ if and only if $(n,6)=1$ and $m\equiv 2k$ or $m\equiv k$ or $k\equiv 0\bmod$ $n$.
\end{itemize}
\end{theorem}

By part (a) $\beta(G_n(m,k)^\mathrm{ab})>0$ if and only if $n\equiv 0\bmod$ $6$ and $m\equiv 2k\bmod$ $6$. We now show that $\beta(H(n,m)^\mathrm{ab})=2$ in these cases (so, in particular, $G_n(m,k)$ is not a connected LOG group).

\begin{lemma}\label{lem:CHRrank}
Suppose $n \geq 1$, $0\leq m,k<n$, $(n,m,k)=1$, let $f(t)=t^m-t^k+1$, $g(t)=t^n-1$ and suppose $n\equiv 0\bmod$ $6$ and $m\equiv 2k\bmod$ $6$. Then $z(t)=(f(t),g(t))=\Phi_6(t)$, and $F(t)=f(t)/\Phi_6(t)$ has no root of modulus 1, and hence $\rho=n-2$ so $\beta(H(n,m)^\mathrm{ab})=2$.
\end{lemma}

\begin{proof}
By \cite[Theorem 3]{Selmer} the polynomial $f(t)=(t^2-t+1)F(t)$ where $F(t)$ has no roots of modulus 1 (see also \cite[Theorem 3]{Ljunggren} or \cite{Tverberg}). Therefore $z(t)=(f(t),g(t))=(t^{2}-t+1,t^n-1)=\Phi_6(t)$, which is of degree 2, and the result follows.
\end{proof}

We require the following lemma, which relies on Theorem \ref{thm:WilliamsOdoni}(b) in an essential way.

\begin{lemma}\label{lem:Gn(m,k)3cases}
Let $n \geq 1$, $0\leq m,k<n$, $(n,m,k)=1$, $m\equiv 2k\bmod$ $6$, $n=ab$ where $a=2^r3^s$, $r,s\geq 1$, $(b,6)=1$. If $\gamma_\rho=1$ then $\mathrm{Res}(f,\Phi_d)=\mathrm{Res}(\Phi_6,\Phi_d)$ for all $d|n$, $d\neq 6$ and ($m\equiv k\bmod$ $b$ or $m\equiv 2k\bmod$ $b$ or $k\equiv 0\bmod$ $b$).
\end{lemma}

\begin{proof}
Suppose $\gamma_\rho=1$. By Lemma \ref{lem:CHRrank} $z(t)=\Phi_6(t)$, so $G(t)=\prod_{d|n, d\neq 6} \Phi_d(t)$ and by Theorem \ref{thm:SNFresultant} $\mathrm{Res}(F,G)=1$. Therefore $\mathrm{Res}(F,\Phi_d)=1$ for all $d|n$, $d\neq 6$, and thus $\mathrm{Res}(f,\Phi_d)=\mathrm{Res}(\Phi_6,\Phi_d)\cdot \mathrm{Res}(F,\Phi_d) =\mathrm{Res}(\Phi_6,\Phi_d)$ for all $d|n$, $d\neq 6$. If $b=1$ then $k\equiv 0\bmod$ $b$, so assume $b>1$.

Now
\[t^b-1 =\prod_{d|b}\Phi_d(t) | \prod_{d|ab,d\neq 6} \Phi_d(t) =G(t)\]
so
\[\mathrm{Res}(F(t),t^b-1) | \mathrm{Res}(F,G)= 1.\]
Also
\begin{alignat*}{1}
\mathrm{Res}(f(t),t^b-1)
&= \mathrm{Res}(\Phi_6(t),t^b-1)\cdot \mathrm{Res}(F(t),t^b-1)\\
&= \mathrm{Res} (\Phi_6,\prod_{d|b} \Phi_d) \cdot \mathrm{Res}(F(t),t^b-1)\\
&= \left(\prod_{d|b} \mathrm{Res} (\Phi_6,\Phi_d) \right) \cdot \mathrm{Res}(F(t),t^b-1).
\end{alignat*}
If $d|b$ then $(d,6)=1$ so $\mathrm{Res} (\Phi_6,\Phi_d)=1$ by \cite[Theorem 3]{Apostol} so $\mathrm{Res}(f(t),t^b-1)=\mathrm{Res}(F(t),t^b-1)=1$ so by Theorem \ref{thm:WilliamsOdoni}(b) $m\equiv k$ or $m\equiv 2k$ or $k\equiv 0\bmod$ $b$.
\end{proof}

Lemma \ref{lem:Gn(m,k)3cases} allows us to prove that strongly irreducible groups $G_n(m,k)$ are not LOG groups:

\begin{theorem}\label{thm:stronglyyirredCHRnotLOG}
Let $n \geq 1$, $0\leq m,k<n$, $(n,m,k)=1$, $(n,k)>1$, $(n,m-k)>1$. Then $G_n(m,k)^\mathrm{ab} \not \cong \Z^2$, and hence $G_n(m,k)$ is not a LOG group.
\end{theorem}

\begin{proof}
Suppose for contradiction that $G_n(m,k)^\mathrm{ab}\cong \Z^2$. Then by Theorem \ref{thm:WilliamsOdoni}(a) $m\equiv 2k\bmod$ $6$ and $n=2^r3^sb$ for some $r,s\geq 1$ and $(b,6)=1$, and by Lemma \ref{lem:Gn(m,k)3cases} $m\equiv k\bmod$ $b$ or $m\equiv 2k\bmod$ $b$ or $k\equiv 0\bmod$ $b$. In particular, $n,m$ are even, so since $(n,m,k)=1$, $k$ is odd. If $3|k$ then $m\equiv 2k\bmod$ $6$ implies $3|(n,m,k)=1$, a contradiction. Thus $(k,6)=1$. In the same way $m-k$ is odd and $3$ does not divide $m-k$. Since $(n,k)>1$ there is a prime divisor $p\geq 5$ of $(n,k)$ and since $(n,m-k)>1$ there is a prime divisor $q\geq 5$ of $(n,m-k)$. If $m\equiv k\bmod$ $b$ or $m\equiv 2k\bmod$ $b$ then $p|m$, so $p|(n,m,k)=1$, a contradiction. If $k\equiv 0\bmod$ $b$ then $q|k$ so $q|(n,m-k,k)=(n,m,k)=1$, a contradiction.
\end{proof}

\section{Gilbert-Howie groups as LOG groups}\label{sec:GH}

By Theorem \ref{thm:stronglyyirredCHRnotLOG} if $G_n(m',k)$ is a LOG group, then it is isomorphic to some Gilbert-Howie group $H(n,m)$ so it remains to consider these groups, whose representer polynomials are the trinomials $f(t)=t^m-t+1$. Throughout this section we let $\gamma_\rho$ denote the last non-zero determinantal divisor of the $n\times n$ circulant matrix $C$ associated with $f$.
This class of groups contains the Sieradski groups $S(2,n)=H(n,2)$, which have free abelianisation if and only if $n\equiv 0\bmod$ $6$, in which case $S(2,n)^\mathrm{ab}\cong \Z^2$ (see \cite[page 236]{JohnsonOdoni} or \cite[Lemma 9]{COS}). We show that in this case $S(2,n)$ is a LOG group.

\begin{theorem}\label{thm:sieradskiasLOG}
The Sieradski group $S(2,6l)$ ($l\geq 1$) is a LOG group with LOG presentation
\[ S(2,6l)= \onetwopres{a_i,b_i\ (0\leq i<2l)}
{a_{2j+1}=b_{2j}^{-1}a_{2j}b_{2j}, a_{2j+1}=b_{2j+1}^{-1}a_{2j+2}b_{2j+1},}
{b_{2j}=a_{2j+1}^{-1}b_{2j+1}a_{2j+1}, b_{2j+2}=a_{2j+2}^{-1}b_{2j+1}a_{2j+2}\ (0\leq j<l)}\]
and where the corresponding LOG has two components.
\end{theorem}

\begin{proof}
In this proof the index $\alpha$ ranges over the integers $0,\ldots, 6l-1$, the index $j$ ranges over the integers $0,\ldots ,l-1$ and the index $i$ ranges over $0,\ldots ,2l-1$. The subscripts of the $y$ generators are to be taken $\bmod 6l$ and the subscripts of $a,b,c$ generators are to be taken $\bmod$ $ 2l$. Now
\begin{alignat*}{1}
S(2,6l)
&= \pres{y_\alpha\ (0\leq \alpha<6l)}{y_\alpha y_{\alpha+2}=y_{\alpha+1}\ (0\leq \alpha<6l)}\\
&=\onethreepres{y_\alpha}
{y_{6j}y_{6j+2}=y_{6j+1}, y_{6j+3}y_{6j+5}=y_{6j+4},}
{y_{6j+2}y_{6j+4}=y_{6j+3}, y_{6j+5}y_{6j+7}=y_{6j+6},}
{y_{6j+1}y_{6j+3}=y_{6j+2}, y_{6j+4}y_{6j+6}=y_{6j+5}}\\
&=\onefivepres{y_\alpha,a_i,b_i,c_i}
{y_{6j+2}=y_{6j}^{-1}y_{6j+1}, y_{6j+5}^{-1}=y_{6j+4}^{-1}y_{6j+3},}
{y_{6j+2}=y_{6j+3}y_{6j+4}^{-1}, y_{6j+5}^{-1}=y_{6j+7}y_{6j+6}^{-1},}
{y_{6j+2}=y_{6j+1}y_{6j+3}, y_{6j+5}^{-1}=y_{6j+6}^{-1}y_{6j+4}^{-1},}
{a_{2j}=y_{6j},b_{2j}=y_{6j+1},c_{2j}=y_{6j+2},}
{a_{2j+1}=y_{6j+3}^{-1},b_{2j+1}=y_{6j+4}^{-1},c_{2j+1}=y_{6j+5}^{-1}}
\\
&=\onethreepres{a_i,b_i,c_i}
{c_{2j}=a_{2j}^{-1}b_{2j}, c_{2j+1}=b_{2j+1}a_{2j+1}^{-1},}
{c_{2j}=a_{2j+1}^{-1}b_{2j+1}, c_{2j+1}=b_{2j+2}a_{2j+2}^{-1},}
{c_{2j}=b_{2j}a_{2j+1}^{-1}, c_{2j+1}=a_{2j+2}^{-1}b_{2j+1}}
\\
&=\onetwopres{a_i,b_i}
{b_{2j}a_{2j+1}^{-1}=a_{2j}^{-1}b_{2j}, a_{2j+2}^{-1}b_{2j+1}=b_{2j+1}a_{2j+1}^{-1},}
{b_{2j}a_{2j+1}^{-1}=a_{2j+1}^{-1}b_{2j+1}, a_{2j+2}^{-1}b_{2j+1}=b_{2j+2}a_{2j+2}^{-1}}
\\
&=\onetwopres{a_i,b_i}
{a_{2j+1}=b_{2j}^{-1}a_{2j}b_{2j}, a_{2j+1}=b_{2j+1}^{-1}a_{2j+2}b_{2j+1},}
{b_{2j}=a_{2j+1}^{-1}b_{2j+1}a_{2j+1}, b_{2j+2}=a_{2j+2}^{-1}b_{2j+1}a_{2j+2}}.
\end{alignat*}
The LOG has two components since $S(2,n)^\mathrm{ab}\cong \Z^2$.
\end{proof}

Figure \ref{fig:LOG} shows the LOG corresponding to the LOG presentation in Theorem \ref{thm:sieradskiasLOG} for the group $S(2,12)$. As an immediate corollary we have:

\begin{corollary}\label{cor:sieradski}
The Sieradski group $S(2,n)$ is a LOG group if and only if $6|n$.
\end{corollary}

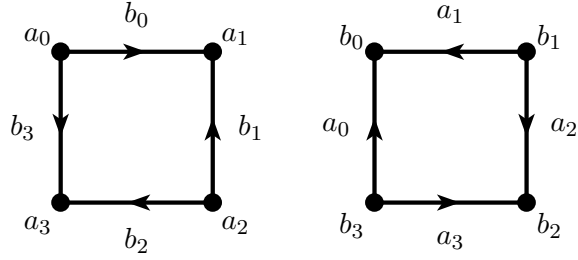
\begin{figure}
\begin{center}
\begin{tabular}{cc}
\psset{xunit=1cm,yunit=1cm,algebraic=true,dimen=middle,dotstyle=o,dotsize=7pt 0,linewidth=1.6pt,arrowsize=3pt 2,arrowinset=0.25}
\begin{pspicture*}(0,0)(4,4)
\psdots[dotstyle=*](1,1)
\psdots[dotstyle=*](1,3)
\psdots[dotstyle=*](3,3)
\psdots[dotstyle=*](3,1)
\psline[ArrowInside=->,ArrowInsidePos=0.5](1,3)(3,3)
\psline[ArrowInside=->,ArrowInsidePos=0.5](1,3)(1,1)
\psline[ArrowInside=->,ArrowInsidePos=0.5](3,1)(1,1)
\psline[ArrowInside=->,ArrowInsidePos=0.5](3,1)(3,3)
\rput(0.7,3.2){$a_0$}
\rput(3.3,3.2){$a_1$}
\rput(3.3,0.7){$a_2$}
\rput(0.7,0.7){$a_3$}
\rput(0.5,2){$b_3$}
\rput(3.5,2){$b_1$}
\rput(2,0.5){$b_2$}
\rput(2,3.5){$b_0$}
\end{pspicture*}
%
%
\psset{xunit=1cm,yunit=1cm,algebraic=true,dimen=middle,dotstyle=o,dotsize=7pt 0,linewidth=1.6pt,arrowsize=3pt 2,arrowinset=0.25}
\begin{pspicture*}(0,0)(4,4)
\psdots[dotstyle=*](1,1)
\psdots[dotstyle=*](1,3)
\psdots[dotstyle=*](3,3)
\psdots[dotstyle=*](3,1)
\psline[ArrowInside=->,ArrowInsidePos=0.5](3,3)(1,3)
\psline[ArrowInside=->,ArrowInsidePos=0.5](1,1)(1,3)
\psline[ArrowInside=->,ArrowInsidePos=0.5](1,1)(3,1)
\psline[ArrowInside=->,ArrowInsidePos=0.5](3,3)(3,1)
\rput(0.7,3.2){$b_0$}
\rput(3.3,3.2){$b_1$}
\rput(3.3,0.7){$b_2$}
\rput(0.7,0.7){$b_3$}
\rput(0.5,2){$a_0$}
\rput(3.5,2){$a_2$}
\rput(2,0.5){$a_3$}
\rput(2,3.5){$a_1$}
\end{pspicture*}
\end{tabular}
\end{center}
  \caption{Labelled Oriented Graph corresponding to the Sieradski group $S(2,12)$.\label{fig:LOG}}
\end{figure}

We conjecture that these Sieradski groups are the only Gilbert-Howie groups (and hence, by Theorem \ref{thm:stronglyyirredCHRnotLOG}, the only groups of Fibonacci type $G_n(m,k)$) that are LOG groups.

\begin{conjecture}\label{conj:H(n,m)freeabelian}
Let $n \geq 1$, $2\leq m<n$, $n\equiv 0\bmod$ $ 6$ and $m\equiv 2\bmod$ $ 6$. Then $\gamma_\rho=1$ if and only if $m=2$. (That is, $H(n,m)^\mathrm{ab}\cong \Z^2$ if and only if $H(n,m)=H(n,2)=S(2,n)$.)
\end{conjecture}

The forward implication is well known and easy. For the converse, in Theorem \ref{thm:finitelymanyfreeabelianGH} we show that for fixed $m$, there can be at most finitely many counterexamples.

\begin{theorem}\label{thm:finitelymanyfreeabelianGH}
Fix $m\geq 8$ where $m\equiv 2\bmod$ $ 6$. Then there exist at most finitely many integers $n$ with $n\equiv 0\bmod$ $ 6$ such that $\gamma_\rho=1$ (that is, for which $H(n,m)^\mathrm{ab}\cong \Z^2$).
\end{theorem}

\begin{proof}
Here $f(t)=t^m-t+1$, $g(t)=t^n-1$, $\gamma_\rho=\mathrm{Res}(F,G)$ where, for $n\equiv 0 \bmod$ $ 6$, $F(t)=f(t)/\Phi_6(t)$, $G(t)=g(t)/\Phi_6(t)$. By \cite[Theorem 1]{Selmer} $F(t)$ is irreducible and not cyclotomic so Corollary \ref{cor:finitelymanyfreeabelian} implies that there are at most finitely many $n$ for which $\gamma_\rho=1$, as required.
\end{proof}

In Theorem \ref{thm:GilbertHowieFreeAbelianmod6b} and Corollary \ref{cor:H(n,m)freeabeliann=6bor12bor24b} we provide further evidence for Conjecture \ref{conj:H(n,m)freeabelian}.
We need the following lemma.

\begin{lemma}\label{lem:GilbertHowieFreeAbelian(m,n)=2}
Suppose $m\equiv 2\bmod$ $ 6$, $n\equiv 0\bmod$ $ 6$.  If $\gamma_\rho=1$ then $(m,n)=2$.
\end{lemma}

\begin{proof}
By Lemma \ref{lem:CHRrank} $z(t)=\Phi_6(t)$, and so $G(t)=\prod_{d|n, d\neq 6} \Phi_d(t)$.
Suppose for contradiction that $\gamma_\rho=1$ and $(m,n)\neq 2$. Then $\mathrm{Res}(F,G)=1$, by Theorem \ref{thm:SNFresultant}.

Let $\delta = (m,n)$, then $\delta$ is even, $\delta\neq 2$, $6\nmid \delta$. Since $\delta | n$ we have $\Phi_\delta(t) | G(t)$, and hence $\mathrm{Res}(F,\Phi_\delta)$ divides $\mathrm{Res}(F,G)=1$. Now
\begin{alignat*}{1}
\mathrm{Res}(f,\Phi_\delta)
&= \mathrm{Res}(\Phi_6(t)F(t),\Phi_\delta(t))\\
&= \mathrm{Res}(\Phi_6,\Phi_\delta)\cdot \mathrm{Res}(F,\Phi_\delta).
\end{alignat*}
Since neither $\delta/6$ nor $6/\delta$ are prime powers, $\mathrm{Res}(\Phi_6,\Phi_\delta)=1$ (see \cite[Theorem 4]{Apostol} or \cite{Lehmer,Diederichsen}) so $\mathrm{Res}(f,\Phi_\delta)= \mathrm{Res}(F,\Phi_\delta)$. On the other hand,
\begin{alignat*}{1}
\mathrm{Res} (f,\Phi_\delta)
&= \prod_{\Phi_\delta(\lambda)=0}(\lambda^m-\lambda+1)\\
&= \prod_{\Phi_\delta(\lambda)=0}(1-\lambda+1)\quad \mathrm{since}~\delta|m\\
&= \mathrm{Res} (t-2,\Phi_\delta(t))=\Phi_\delta(2).
\end{alignat*}
But $\delta\geq 4$ so, by \cite[Corollary 9]{Roitman}, $\Phi_\delta (2)>2^{\sqrt{\delta}/4}>1$. Thus $\mathrm{Res} (F,\Phi_\delta)=\mathrm{Res} (f,\Phi_\delta)>1$, a contradiction.
\end{proof}

\begin{theorem}\label{thm:GilbertHowieFreeAbelianmod6b}
Suppose $m\equiv 2\bmod$ $ 6$, $n=ab$ where $a=2^r\cdot 3^s$ where $r,s\geq 1$, $(b,6)=1$.  If $\gamma_\rho=1$ then $\mathrm{Res}(f,\Phi_d)=\mathrm{Res}(\Phi_6,\Phi_d)$ for all $d|n$, $d\neq 6$ and, moreover, $m\equiv 2\bmod$ $ 6b$ if $r=1$ and $m\equiv 2\bmod$ $ 12b$ if $r\geq 2$.
\end{theorem}

\begin{proof}
If $r\geq 2$ then by Lemma \ref{lem:GilbertHowieFreeAbelian(m,n)=2} we may assume $m\equiv 2\bmod$ $ 4$, in which case $m\equiv 2\bmod$ $ 12b$ if and only if $m\equiv 2\bmod$ $ 6b$, so it suffices to prove this last condition. If $b=1$ then the hypothesis $m\equiv 2\bmod$ $ 6$ immediately implies the conclusion that $m\equiv 2\bmod$ $ 6b$, so assume $b>1$.

Suppose $\gamma_\rho=1$. Then, by Lemma \ref{lem:Gn(m,k)3cases}, $\mathrm{Res}(f,\Phi_d)=\mathrm{Res}(\Phi_6,\Phi_d)$ for all $d|n$, $d\neq 6$ and either $m\equiv 1$ or $2\bmod$ $ b$.

\noindent \textbf{Case 1: $m\equiv 1\bmod$ $ b$}. Here $m\equiv 1+\alpha b\bmod$ $ n$ for some $0\leq \alpha<a$. We claim that $(2^b+1)/3$ divides $\mathrm{Res}(F,G)=1$, a contradiction (since $b>1$). Observe that $m\equiv 2\bmod$ $ 6$ implies that $\alpha b\equiv 1\bmod$ $ 6$, so $(\alpha,6)=1$ and, in particular, $\alpha$ is odd.

Since $b$ is odd, $t^b+1=\prod_{d|b} \Phi_{2d}(t)$. Since $\prod_{d|b} \Phi_{2d}(t)| \prod_{d|n, d\neq 6} \Phi_d(t)$, $t^b+1$ divides $G(t)$, and hence $\mathrm{Res}(F(t),t^b+1)$ divides $\mathrm{Res}(F,G)=1$. Now
\begin{alignat*}{1}
\mathrm{Res} (f(t),t^b+1)
&= \mathrm{Res} (t\cdot(t^b)^\alpha-t+1,t^b+1)\\
&= \mathrm{Res} (t\cdot(-1)^\alpha-t+1,t^b+1)\\
&= \mathrm{Res} (1-2t,t^b+1)\\
&= 2^b+1.
\end{alignat*}
On the other hand (recalling that $z(t)=(f(t),g(t))=\Phi_6(t)$, by Lemma \ref{lem:CHRrank}),
\begin{alignat*}{1}
\mathrm{Res} (f(t),t^b+1)
&=\mathrm{Res} (\Phi_6(t),t^b+1)\cdot \mathrm{Res} (F(t),t^b+1)\\
&=\mathrm{Res} (\Phi_6,\prod_{d|b}\Phi_{2d})\cdot \mathrm{Res} (F(t),t^b+1)\\
&=\left(\prod_{d|b}\mathrm{Res} (\Phi_6,\Phi_{2d})\right)\cdot \mathrm{Res} (F(t),t^b+1)\\
&=\mathrm{Res} (\Phi_6,\Phi_{2})\left(\prod_{d|b,d>1}\mathrm{Res} (\Phi_6,\Phi_{2d})\right)\cdot \mathrm{Res} (F(t),t^b+1).
\end{alignat*}
But if $d|b, d>1$ then neither $(2d)/6$ or $6/(2d)$ is a prime power, so in the product each $\mathrm{Res} (\Phi_6,\Phi_{2d})=1$, and hence $\mathrm{Res} (f(t),t^b+1)=\mathrm{Res} (\Phi_6,\Phi_{2}) \cdot \mathrm{Res} (F(t),t^b+1)= 3 \cdot \mathrm{Res} (F(t),t^b+1)$. Thus $\mathrm{Res}(F(t),t^b+1)=(2^b+1)/3$, so $(2^b+1)/3$ divides $\mathrm{Res}(F,G)$ as claimed.

\medskip

\noindent \textbf{Case 2: $m\equiv 2\bmod$ $ b$}. Here $m= 2+\alpha b$ for some $0\leq \alpha<a$. Now $m\equiv 2\bmod$ $ 6$ implies that $\alpha b\equiv 0\bmod$ $ 6$, and hence $6|\alpha$, so $\alpha =6\beta$ for some $0\leq \beta <a/6$. Thus $m\equiv 2+6\beta b\bmod$ $ n$, i.e.\,$m\equiv 2\bmod$ $ 6b$, as claimed.
\end{proof}

\begin{corollary}\label{cor:H(n,m)freeabeliann=6bor12bor24b}
Let $n=6b$ or $n=12b$ or $n=24b$ where $(b,6)=1$, $2\leq m<n$, and $m\equiv 2\bmod$ $ 6$. Then $\gamma_\rho=1$ if and only if $m=2$. (That is, $H(n,m)^\mathrm{ab}\cong \Z^2$ if and only if $H(n,m)=H(n,2)=S(2,n)$.)
\end{corollary}

\begin{proof}
The cases $n=6b$ or $n=12b$ follow immediately from Theorem \ref{thm:GilbertHowieFreeAbelianmod6b} so assume $n=24b$. Factorize $g(t)=(t^{12b}-1)(t^{12b}+1)$. Since $\Phi_6(t)$ divides $(t^{12b}-1)$ the factor $(t^{12b}+1)$ divides $G(t)$, and hence $\mathrm{Res}(f(t),t^{12b}+1)$ divides $\mathrm{Res}(f,G)$. The resultant $\mathrm{Res}(f,G)=\gamma_\rho \cdot \mathrm{Res}(\Phi_6,G)=G(\zeta_6)G(\zeta_6^{-1})=n^2/3$, by an application of l'H\^{o}pital's rule, so it suffices to show  $\mathrm{Res}(f(t),t^{12b}+1)>n^2/3$. By Theorem \ref{thm:GilbertHowieFreeAbelianmod6b} we may assume $m=2+12b$. Therefore
\begin{alignat*}{1}
\mathrm{Res}(f(t),t^{12b}+1)
&= \mathrm{Res}(t^{2}(t^{12b})-t+1,t^{12b}+1)\\
&= \mathrm{Res}(-t^{2}-t+1,t^{12b}+1)\\
&= \left(((-1+\sqrt{5})/2)^{12b} +1\right)\left(((-1-\sqrt{5})/2)^{12b} +1\right)\\
&= 2 + \left((-1-\sqrt{5})/2 \right)^{12b} +\left((-1+\sqrt{5})/2 \right)^{12b}\\
&= 2 + L_{12b}
\end{alignat*}
where $L_{12b}$ is the $12b$-th Lucas number. Since $2+L_{12b}>192b^2=n^2/3$ for all $b\geq 1$ the resultant $\mathrm{Res}(f(t),t^{12b}+1)>n^2/3$, as required.
\end{proof}

\section{Cyclically presented groups and Prishchepov groups as connected LOG groups}\label{sec:knotgroups}

Connected LOG groups abelianize to $\Z$. In this section we obtain necessary conditions for a cyclically presented group $G_n(w)$ with representer polynomial $f(t)$ to abelianize to $\Z$. If $G_n(w)^\mathrm{ab}\cong \Z$ then, in particular, $\beta (G_n(w)^\mathrm{ab})=1$ is odd. The following characterisations follow immediately from Corollary \ref{cor:G_n(w)^ab} (where $\rho$ is the rank, and $\gamma_\rho$ is the last non-zero determinantal divisor of the  $n\times n$ circulant matrix associated with $f$):
\begin{alignat}{1}
\beta(G_n(w)^\mathrm{ab})~\mathrm{is~odd}\Leftrightarrow n-\rho~\mathrm{is~odd}\Leftrightarrow \mathrm{deg}((f(t),g(t)))~\mathrm{is~odd},\label{eq:oddbetti}
\end{alignat}
\begin{alignat}{1}
G_n(w)^\mathrm{ab}\cong \Z\Leftrightarrow n-\rho=1~\mathrm{and}~\gamma_\rho=1\Leftrightarrow \mathrm{deg}((f(t),g(t)))=1~\mathrm{and}~\gamma_\rho=1.\label{eq:G_n(w)=Z}
\end{alignat}
We refine (\ref{eq:oddbetti}) slightly:

\begin{prop}\label{prop:oddbetti}
Let $f(t)\in \Z[t]$, $g(t)=t^n-1$. Then $n-\rho$ is odd if and only if one of the following holds:
\begin{itemize}
  \item[(a)] $n$ is odd and $f(1)=0$; or
  \item[(b)] $n$ is even and either
  \begin{itemize}
    \item[(i)] $f(1)=0$ and $f(-1)\neq 0$; or
    \item[(ii)] $f(-1)=0$ and $f(1)\neq 0$.
  \end{itemize}
\end{itemize}
\end{prop}

\begin{proof}
This holds since the roots of $g(t)$ arise in complex conjugate pairs and the only real roots of $g(t)$ are $1$ and (when $n$ is even) $-1$.
\end{proof}

We now give necessary and sufficient conditions for (\ref{eq:G_n(w)=Z}) to hold:

\begin{theorem}\label{thm:G^ab=Z}
Let $f(t)\in \Z[t]$, $g(t)=t^n-1$ and let $\nu=\mathrm{max}\{ d\ |\ f(t)\in \Z[t^d] \}$. Then $n-\rho=1$ and $\gamma_\rho=1$ if and only if $(n,\nu)=1$ and one of the following holds:
\begin{itemize}
  \item[(a)] $n$ is odd, $f(1)=0$, $\left( f(t), g(t) \right)=t-1$, and $\mathrm{Res}\left( f(t)/(t-1), \sum_{i=0}^{n-1} t^i \right)=1$;
  \item[(b)] $n$ is even, $f(1)=0$, $|f(-1)|=2$, $\left( f(t), g(t) \right)=t-1$, and $\mathrm{Res}\left( f(t)/(t-1), \sum_{i=0}^{n-1} t^i \right)=1$;
  \item[(c)] $n$ is even, $f(-1)=0$, $|f(1)|=2$, $\left( f(t), g(t)\right)=t+1$, and $\mathrm{Res}\left( f(t)/(t+1), \sum_{i=0}^{n-1} (-t)^i \right)=1$ and $\mathrm{Res}(f(t),t^c-1)=2$ where $c$ is the largest odd divisor of $n$.
\end{itemize}
\end{theorem}

\begin{proof}
Let $z(t)=(f(t),g(t))$. If $(n,\nu)=1$ and any of (a),(b),(c) hold then $z(t)=t-\epsilon$ ($\epsilon=1$ in cases (a),(b), $\epsilon=-1$ in case~(c)), and $F(t)=f(t)/(t-\epsilon)$, $G(t)=(t^n-1)/(t-\epsilon)=\sum_{i=0}^{n-1} (\epsilon t)^i$. Thus $n-\rho =\mathrm{deg} (z(t))=1$, $\gamma_\rho=\mathrm{Res} (F,G)=1$, as required.

Suppose then $n-\rho=1$ and $\gamma_\rho=1$. Then $\mathrm{deg}(z(t))=1$ and $\mathrm{Res}(f(t)/z(t),g(t)/z(t))=1$ by Theorem \ref{thm:SNFresultant}. Let $\delta=(n,\nu)$. Then $f(t)=\bar{f}(t^\delta), g(t)=\bar{g}(t^\delta)$ for some $\bar{f},\bar{g}\in \Z[t]$. Thus $z(t)=(f(t),g(t))=(\bar{f}(t^\delta),\bar{g}(t^\delta))=\bar{z}(t^\delta)$, say. Then $n-\rho=1$ implies $\mathrm{deg}(z(t))=1$, so $\bar{z}(s)$ is of degree 1 and $\delta=1$, i.e.\,$(n,\nu)=1$.

If $n$ is odd  then Proposition \ref{prop:oddbetti} implies $f(1)=0$, and so $(t-1)|z(t)$, but since $\mathrm{deg}(z(t))=1$ we have $z(t)=t-1$, $F(t)=f(t)/(t-1)$, $G(t)=g(t)/(t-1)=\sum_{i=0}^{n-1}t^i$, so part (a) follows.

Assume then $n$ is even. Then by Proposition \ref{prop:oddbetti}(b) there is a unique $\epsilon \in\{1,-1\}$ such that $f(\epsilon)=0, f(-\epsilon)\neq 0$. Then $(t-\epsilon)$ divides $z(t)$, but since $\mathrm{deg}(z(t))=1$ we have $z(t)=t-\epsilon$, $F(t)=f(t)/(t-\epsilon), G(t)=g(t)/(t-\epsilon)=\sum_{i=0}^{n-1}(\epsilon t)^i$. Now $g(-\epsilon)=0$ so $G(-\epsilon)=0$ and hence $|F(-\epsilon)|$ divides $\mathrm{Res}(F,G)|=1$. Thus $|F(-\epsilon)|=1$. But $|f(-\epsilon)|=|(-\epsilon-\epsilon)F(-\epsilon)|=2|F(-\epsilon)|$ so $\{ |f(\epsilon)|,|f(-\epsilon)|\}=\{0,2\}$, and the proof of part (b) is complete. To complete the proof of part (c) it remains to show  $\mathrm{Res}(f(t),t^c-1)|=2$ where $c$ is the largest odd divisor of $n$. Now $\mathrm{Res}(f(t)/(t+1), (t^n-1)/(t+1))=1$ so (since $(t^c-1)$ divides $(t^n-1)/(t+1)$) we have $\mathrm{Res}(f(t)/(t+1), t^c-1)=1$ and hence $\mathrm{Res}( (f(t)/(t+1),\Phi_\delta(t))=1$ for all divisors $\delta$ of $c$. Then if $\delta>1$
\begin{alignat*}{1}
\mathrm{Res} (f, \Phi_\delta ) &= \mathrm{Res} ((t+1), \Phi_\delta (t)) \cdot \mathrm{Res} (f(t)/(t+1), \Phi_\delta (t)) \\
&\qquad \qquad = \mathrm{Res} ((t+1), \Phi_\delta (t)) = \Phi_\delta (-1)= \Phi_{2\delta} (1)=1
\end{alignat*}
since $2\delta$ is not a prime power. Hence
\[\mathrm{Res}(f(t), t^c-1)= \prod_{\delta |c} \mathrm{Res}(f, \Phi_\delta)= \mathrm{Res}(f(t),t-1) \cdot \mathrm{Res}(f, \prod_{\delta |c, \delta>1} \Phi_\delta)
= f(1)\cdot 1=2
\]
as required.
\end{proof}

\begin{corollary}\label{cor:positiveword}
Suppose $w$ is a positive word of length at least 3. Then $G_n(w)^\mathrm{ab} \not \cong \Z$.
\end{corollary}

\begin{proof}
Since $w$ is positive of length at least 3, $f(1)\geq 3$ so Theorem \ref{thm:G^ab=Z} implies that if $G_n(w)^\mathrm{ab} \cong \Z$ then $|f(1)|=2$, a contradiction.
\end{proof}

If $w$ is a positive word of length 1 then $G_n(w)$ is trivial. If $w$ is a positive word of length 2, then either $G_n(w)=G_n(x_0^2)\cong \Z_2^n$ or $G_n(w)=G_n(x_0x_k)$ for some $1\leq k<n$, which is free of rank $(n,k)$ if $n/(n,k)$ is even and is the free product of $(n,k)$ copies of $\Z_2$ if $n/(n,k)$ is odd. Therefore, for a positive word $w$,  $G_n(w)^\mathrm{ab}\cong \Z$ if and only if $G_n(w)=G_n(x_0x_k)$ where $n$ is even and $(n,k)=1$, in which case $G_n(w)\cong \Z$ (a connected LOG group).

By \cite[Theorem 1]{SV} the natural HNN extension (see \cite{SV} for the definition) of a cyclically presented group $G_n(w)$ with representer polynomial $f(t)$ is a $k$-knot group ($k\geq 3$) if and only if $|f(1)|=1$. It therefore follows from this and Theorem \ref{thm:G^ab=Z} that a cyclically presented group $G_n(w)$ and its natural HNN extension $\mathcal{G}_n(w)$ cannot both be $k$-knot groups ($k\geq 3$).

Theorem \ref{thm:G^ab=Z} can be applied to particular classes of cyclically presented groups. To illustrate this, we apply it to the Prishchepov groups $P(r,n,k,s,q)$ which have representer polynomials
\[ f(t)=1+t^q+\ldots + t^{(r-1)q}-t^{k-1}(1+t^q+\ldots + t^{(s-1)q}).\]
\begin{corollary}\label{cor:prishchepov}
Suppose that $P(r,n,k,s,q)^\mathrm{ab}\cong \Z$ (which holds, in particular, for connected LOG groups), where $n \geq 1$, $0\leq k, q<n$ and $r,s\geq 1$, $r\neq s$. Then $n$ is even, $(n,k-1,q)=1$, $|r-s|=2$, $q$ is odd, $P(r,c,k,s,q)^\mathrm{ab}\cong \Z_2$  for the largest odd divisor $c$ of $n$, and either (i) $s$ is even; or (ii) $s$ is odd and $k$ is odd.
\end{corollary}

\begin{proof}
Suppose $P(r,n,k,s,q)^\mathrm{ab}\cong \Z$. If $n$ is odd then Theorem \ref{thm:G^ab=Z} implies $f(1)=0$, i.e.\,$r-s=0$, a contradiction to $r\neq s$. Thus $n$ is even. Then $(n,k-1,q)$ divides $(n,\nu)$ so $(n,k-1,q)=1$. Again $f(1)=r-s\neq 0$ so Theorem \ref{thm:G^ab=Z} implies $f(1)=2, f(-1)=0$ and $2=\mathrm{Res}(f(t),t^c-1)=|P(r,c,k,s,q)^\mathrm{ab}|$ (and so $P(r,c,k,s,q)^\mathrm{ab}\cong \Z_2$) for the largest odd divisor $c$ of $n$. But $f(1)=r-s$ so $|r-s|=2$, and $f(-1)=0$ if and only if $q$ is odd and either $r,s$ are both even or $r,s,k$ are all odd, and the result follows.
\end{proof}

If $n,s$ are even and $(n,(s+1)q)=1$ then $P(s+2,n,q+1,s,q)^\mathrm{ab}\cong G_n(x_0x_{(s+1)q})^\mathrm{ab} \cong G_n(x_0x_{1})^\mathrm{ab}\cong \Z$. In this case the representer polynomial $f(t)=1+t^{(s+1)q}$. However, there are examples of groups $P(r,n,k,s,q)$ that abelianize to $\Z$ where $f(t)$ is more complicated. For example the group $P(4,10,3,2,7)$ with $f(t)=1+t-t^2+t^4+t^7-t^9=-(t+1)(t^8-t^7-t^3+t^2-1)$. Determining precisely which groups $P(r,n,k,s,q)$ abelianize to $\Z$ is a topic for future research.

\section*{Acknowledgements}

Both authors thank the Department of Computer Science and the Department of Mathematics at the University of Pisa for their hospitality during a visit in November 2019 when part of this work was carried out.

  \textsc{Department of Mathematics and Systems Analysis, Aalto University, P.O. Box 11000 (Otakaari 24), FI-00076 AALTO, Finland.}\par\nopagebreak
  \textit{E-mail address}, \texttt{vanni.noferini@aalto.fi}

  \textsc{Department of Mathematical Sciences, University of Essex, Wivenhoe Park, Colchester, Essex CO4 3SQ, UK.}\par\nopagebreak
  \textit{E-mail address}, \texttt{gerald.williams@essex.ac.uk}

\begin{thebibliography}{10}

\bibitem{Apostol}
T.~M. {Apostol}.
\newblock {Resultants of cyclotomic polynomials.}
\newblock {\em {Proc. Am. Math. Soc.}}, 24:457--462, 1970.

\bibitem{BakerBook}
Alan Baker.
\newblock {\em Transcendental number theory}.
\newblock Cambridge Mathematical Library. Cambridge University Press,
  Cambridge, second edition, 1990.

\bibitem{BardakovVesnin}
V.~G. Bardakov and A.~Yu. Vesnin.
\newblock On a generalization of {F}ibonacci groups.
\newblock {\em Algebra Logika}, 42(2):131--160, 255, 2003.

\bibitem{BW2}
William~A. Bogley and Gerald Williams.
\newblock Coherence, subgroup separability, and metacyclic structures for a
  class of cyclically presented groups.
\newblock {\em J. Algebra}, 480:266--297, 2017.

\bibitem{BorweinErdelyi}
Peter Borwein and Tam\'{a}s Erd\'{e}lyi.
\newblock {\em Polynomials and polynomial inequalities}, volume 161 of {\em
  Graduate Texts in Mathematics}.
\newblock Springer-Verlag, New York, 1995.

\bibitem{CampbellRobertson75}
C.M. {Campbell} and E.F. {Robertson}.
\newblock {On a class of finitely presented groups of Fibonacci type.}
\newblock {\em {J. Lond. Math. Soc., II. Ser.}}, 11:249--255, 1975.

\bibitem{CHRsurvey}
Alberto {Cavicchioli}, Friedrich {Hegenbarth}, and Du\v{s}an {Repov\v{s}}.
\newblock {On manifold spines and cyclic presentations of groups.}
\newblock In {\em {Knot theory. Proceedings of the mini-semester, Warsaw,
  Poland, July 13--August 17, 1995}}, pages 49--56. Warszawa: Polish Academy of
  Sciences, Institute of Mathematics, 1998.

\bibitem{COS}
Alberto Cavicchioli, E.~A. O'Brien, and Fulvia Spaggiari.
\newblock On some questions about a family of cyclically presented groups.
\newblock {\em J. Algebra}, 320(11):4063--4072, 2008.

\bibitem{BainsonChinyere}
Ihechukwu Chinyere and Bernard~Oduoku Bainson.
\newblock Perfect {P}rishchepov groups.
\newblock {\em J. Algebra}, 588:515--532, 2021.

\bibitem{Cremona}
J.~E. {Cremona}.
\newblock {Unimodular integer circulants.}
\newblock {\em {Math. Comput.}}, 77(263):1639--1652, 2008.

\bibitem{Diederichsen}
Fritz-Erdmann {Diederichsen}.
\newblock {\"Uber die Ausreduktion ganzzahliger Gruppendarstellungen bei
  arithmetischer \"Aquivalenz.}
\newblock {\em {Abh. Math. Semin. Univ. Hamb.}}, 13:357--412, 1940.

\bibitem{GilbertHowie}
N.~D. {Gilbert} and James {Howie}.
\newblock {LOG groups and cyclically presented groups.}
\newblock {\em {J. Algebra}}, 174(1):118--131, 1995.

\bibitem{HarlanderRosebrock}
Jens {Harlander} and Stephan {Rosebrock}.
\newblock {Aspherical word {L}abeled {O}riented {G}raphs and cyclically
  presented groups.}
\newblock {\em {J. Knot Theory Ramifications}}, 24(5):7, 2015.

\bibitem{HowieRibbonDiscComp85}
James {Howie}.
\newblock {On the asphericity of ribbon disc complements.}
\newblock {\em {Trans. Am. Math. Soc.}}, 289:281--302, 1985.

\bibitem{HowieWilliams}
James {Howie} and Gerald {Williams}.
\newblock {Tadpole labelled oriented graph groups and cyclically presented
  groups.}
\newblock {\em {J. Algebra}}, 371:521--535, 2012.

\bibitem{JohnsonBook}
D.~L. {Johnson}.
\newblock {\em {Topics in the theory of group presentations}}, volume~42.
\newblock Cambridge University Press, Cambridge. London Mathematical Society,
  London, 1980.

\bibitem{JohnsonMawdesley}
D.~L. Johnson and H.~Mawdesley.
\newblock Some groups of {F}ibonacci type.
\newblock {\em J. Austral. Math. Soc.}, 20(2):199--204, 1975.

\bibitem{JohnsonOdoni}
D.~L. {Johnson} and R.~W.~K. {Odoni}.
\newblock {Some results on symmetrically-presented groups.}
\newblock {\em {Proc. Edinb. Math. Soc., II. Ser.}}, 37(2):227--237, 1994.

\bibitem{Lehmer}
E.~T. {Lehmer}.
\newblock {A numerical function applied to cyclotomy.}
\newblock {\em {Bull. Am. Math. Soc.}}, 36:291--298, 1930.

\bibitem{Ljunggren}
Wilhelm {Ljunggren}.
\newblock {On the irreducibility of certain trinomials and quadrinomials.}
\newblock {\em {Math. Scand.}}, 8:65--70, 1960.

\bibitem{MKS}
Wilhelm Magnus, Abraham Karrass, and Donald Solitar.
\newblock {\em Combinatorial group theory}.
\newblock Dover Publications, Inc., Mineola, NY, second edition, 2004.

\bibitem{NoferiniWilliamsCompanion}
Vanni Noferini and Gerald Williams.
\newblock Matrices in companion rings, {S}mith forms, and the homology of
  3-dimensional {B}rieskorn manifolds.
\newblock {\em J. Algebra}, 587:1--19, 2021.

\bibitem{Odoni}
R.~W.~K. {Odoni}.
\newblock {Some Diophantine problems arising from the theory of
  cyclically-presented groups.}
\newblock {\em {Glasg. Math. J.}}, 41(2):157--165, 1999.

\bibitem{Prishchepov}
Matve{\u{\i}}~I. {Prishchepov}.
\newblock {Aspherisity, atorisity and symmetrically presented groups.}
\newblock {\em {Commun. Algebra}}, 23(13):5095--5117, 1995.

\bibitem{Roitman}
Moshe {Roitman}.
\newblock {On Zsigmondy primes.}
\newblock {\em {Proc. Am. Math. Soc.}}, 125(7):1913--1919, 1997.

\bibitem{Selmer}
Ernst~S. {Selmer}.
\newblock {On the irreducibility of certain trinomials.}
\newblock {\em {Math. Scand.}}, 4:287--302, 1956.

\bibitem{Sieradski}
Allan~J. {Sieradski}.
\newblock {Combinatorial squashings, 3-manifolds, and the third homology of
  groups.}
\newblock {\em {Invent. Math.}}, 84:121--139, 1986.

\bibitem{Simon}
Jonathan {Simon}.
\newblock {Wirtinger approximations and the knot groups of $F\sp n$ in
  $S\sp{n+2}$.}
\newblock {\em {Pac. J. Math.}}, 90:177--190, 1980.

\bibitem{SV}
Andrzej {Szczepa\'nski} and Andrei {Vesnin}.
\newblock {HNN extension of cyclically presented groups.}
\newblock {\em {J. Knot Theory Ramifications}}, 10(8):1269--1279, 2001.

\bibitem{Tverberg}
Helge Tverberg.
\newblock On the irreducibility of the trinomials {$x^{n}\pm x^{m}\pm 1$}.
\newblock {\em Math. Scand.}, 8:121--126, 1960.

\bibitem{WilliamsCHR}
Gerald Williams.
\newblock The aspherical {C}avicchioli-{H}egenbarth-{R}epov\v{s} generalized
  {F}ibonacci groups.
\newblock {\em J. Group Theory}, 12(1):139--149, 2009.

\bibitem{WilliamsIJNT}
Gerald Williams.
\newblock Unimodular integer circulants associated with trinomials.
\newblock {\em International Journal of Number Theory}, 6(04):869--876, 2010.

\bibitem{Wsurvey}
Gerald Williams.
\newblock Groups of {F}ibonacci type revisited.
\newblock {\em Internat. J. Algebra Comput.}, 22(8):1240002, 19 pages, 2012.

\bibitem{WilliamsLOG}
Gerald Williams.
\newblock Generalized {F}ibonacci groups {$H(r,n,s)$} that are connected
  {L}abelled {O}riented {G}raph groups.
\newblock {\em J. Group Theory}, 22(1):23--39, 2019.

\end{thebibliography}
\end{document}